\documentclass[12pt]{article}
\usepackage{amsmath,amsfonts,amsthm,mathrsfs,indentfirst}
\usepackage[left=3cm,right=3cm,top=2.5cm,bottom=2.5cm]{geometry}
\usepackage{graphics}
\usepackage{authblk}
\usepackage[numbers,sort&compress]{natbib}
\usepackage[hypertex]{hyperref}
\usepackage{xcolor}
\usepackage{colortbl}

\newtheorem{thm}{Theorem}

\newtheorem{lemma}{Lemma}

\newtheorem{defn}{Definition}
\newtheorem{rem}{Remark}
\newtheorem{corollary}{Corollary}

\def\p{\mathcal{P}}

\def\ba{\boldsymbol\alpha}

\def\R{\mathbb{R}}
\def\N{\mathbb{N}}
\def\bp{\mathbf{p}}
\def\X{\mathscr{X}}
\def\Y{\mathscr{Y}}

\def\s{\mathcal{S}}

\def\span{\mathrm{Span}}

\allowdisplaybreaks
\begin{document}
\title{On the Singularity of Multivariate Hermite Interpolation  \thanks{This project is supported by
NNSFC (Nos. 11301053,61033012,11171052,11271060,61272371) and ``the
Fundamental Research Funds for the Central Universities''.}}
\author[a]{Zhaoliang Meng\thanks{Corresponding author: mzhl@
dlut.edu.cn}}
\author[a,b]{Zhongxuan Luo}
\affil[a]{\it \small School of Mathematical Sciences,Dalian
University of Technology, Dalian, 116024, China}  \affil[b]{\it School of Software,Dalian
University of Technology, Dalian, 116620, China}

\maketitle
\begin{abstract}
In this paper we study the singularity of multivariate Hermite
interpolation of type total degree. We present two methods to judge the
singularity of the interpolation schemes considered and by methods
to be developed, we show that all Hermite interpolation of type
total degree  on $m=d+k$ points in $\R^d$ is singular if $d\geq 2k$.
And then we solve the Hermite interpolation problem on $m\leq d+3$
nodes completely. Precisely, all Hermite interpolations of type
total degree on $m\leq d+1$ points with $d\geq 2$ are singular;  only three cases
for
$m=d+2$ and one case for $m=d+3$ can produce regular
Hermite interpolation schemes, respectively. Besides, we also
present a method to compute the interpolation space for Hermite
interpolation of type total degree.
    \\[6pt]
    \textbf{Keywords:} Hermite interpolation; Singularity;
    Interpolation space; Polynomial ideal
\end{abstract}

\section{Introduction}

Let $\Pi^d$ be the space of all polynomials in $d$ variables, and let
$\Pi_n^d$ be the subspace of polynomials of total degree at most $n$.
 Let $\X=\{X_1,X_2,\ldots,X_m\}$ be a set of pairwise distinct
points in $\R^d$ and $\bp=\{p_1,p_2,\ldots,p_m\}$ be a set of $m$ nonnegative integers.
The Hermite interpolation problem to be considered in this paper is described as follows:
Find a (unique) polynomial $f\in \Pi_n^d$ satisfying
\begin{equation}
    \frac{\partial^{\alpha_1+\alpha_2+\ldots+\alpha_d}}{\partial
    x_1^{\alpha_1}\ldots \partial x_d^{\alpha_d}} f(X_q)=c_{q,\ba},\quad
    1\leq q\leq m,\quad 0\leq |\ba|\leq p_q,
\end{equation}
for given values $c_{q,\ba}$, where the numbers $p_q$ and $n$ are assumed to satisfy
\begin{eqnarray}\label{number_condition}
    \binom{n+d}{d}=\sum_{q=1}^m\binom{p_q+d}{d}.
\end{eqnarray}

Following \cite{1992-Lorentz-p-,2000-Lorentz-p167-201}, such kind of
problem is called Hermite interpolation of type total degree. The
interpolation problem $(\bp,\X)$ is called regular if the above
equation has a unique solution for each choice of values
$\{c_{q,\ba},1\leq q\leq m,0\leq |\ba|\leq p_q\}$. Otherwise, the
interpolation problem is singular. As shown in \cite{1995-Gevorgian-p23-35}, the regularity of Hermite
interpolation problem $(\bp,\X)$ implies that it is regular for
almost all $\X\subset\R^d$ with $|\X|=m$.
\begin{defn}[\cite{1995-Gevorgian-p23-35}]
    We say that the interpolation scheme $\bp$ is:
    \begin{itemize}
    \item Regular if the problem $(\bp,\X)$ is regular for all $\X$.
    \item Almost regular if the problem $(\bp,\X)$ is regular for almost all $\X$.
    \item Singular if $(\bp,\X)$ is singular for all $\X$.
    \end{itemize}
\end{defn}

The special case in which the $p_q$ are all the same is called uniform Hermite
interpolation of type total degree. In the case of uniform Hermite
interpolation of type total degree, Eq. \eqref{number_condition} should be
changed to
\begin{eqnarray}\label{number_condition_uniform}
    \binom{n+d}{d}=m\binom{p+d}{d}.
\end{eqnarray}

The research of regularity of multivariate Hermite interpolation is
more difficult than Lagrange case, although the latter is also
difficult. One of the main reasons is that Eq.
\eqref{number_condition} or \eqref{number_condition_uniform} does not
hold in some cases. Up to now, we have known that all the Hermite
interpolation on $m\leq d+1$ points are singular except for Lagrange
interpolation, see
\cite{1992-Lorentz-p-,2000-Lorentz-p167-201,1984-LeM'ehaut'e-p-}.
Besides, no any other results appeared for $m\geq d+2$.
Actually, Hermite interpolation of type total degree on $d+2$ nodes
in $\R^d$ are not necessary singular. For more research of this
area, we can refer to
\cite{sauer1995multivariate,2000-Lorentz-p167-201,1992-Lorentz-p-,lorentz1990bivariate,1984-LeM'ehaut'e-p-,habib1996recursive,2000-Gasca-p23-35,2000-Gasca-p377-410,gasca2000bivariate,gasca1982lagrange,chai2011proper}
and the reference therein.

The main purpose of this paper is to investigate the singularity of
Hermite interpolation for $m=d+k$ with $k=1,2,3$. This paper will propose two
ways to prove singularity for Hermite interpolation of type total degree.
The first one consists of constructing the interpolation space from the view
of polynomial ideal. We state this method in a general way which is also useful for other types of
interpolations. The second one depends on an algorithm (see Theorem \ref{thm:main})
from which we can get a polynomial being a solution of the homogenous
interpolation problem. This method leads to the most general singularity
theorems. To get complete results for  $m=d+k$ with $k=1,2,3$, we employ the
second method to get the results for general cases and employ the first one
for other special cases.

By the presented methods, we show that all
Hermite interpolation of type total degree on $m\leq d+1$ nodes in
$\R^d$ are singular except for Lagrange interpolation; on $m=d+2$
nodes in $\R^d$ are singular except for three cases; on $m=d+3$ nodes are
singular except for one case.
Moreover, we also show all the hermite
interpolation problem of type total degree with $m=d+k$ nodes are
singular for $d\geq 2k$.
The result of $m\leq d+1$ is well known,
but our method is different.
To the best of our knowledge, all the results except
for the case of $m\leq d+1$ seem to new. 

This paper is organized as follows. In section 2, we will consider
the interpolation space satisfying the Hermite interpolation requirement
from the view of polynomial ideal. In this section, Eq.
\eqref{number_condition} is not required and the polynomial space is
not necessary $\Pi_n^d$. In section 3, we consider the singularity
of the Hermite interpolation of type total degree and present the
main results. Finally, in section 4, we conclude our results.

\section{Interpolation Space}
It is well known that polynomial interpolation is closely related to polynomial
ideal. This relation is implied in many early papers and widely employed, for
example
\cite{2001-Sauer-p2293-2309,2000-Moeller-p335-362,1990-Boor-p287-302}. In \cite{2000-Xu-p363-376},
Xu presented a solution to the
Lagrange interpolation problem from the view of polynomial ideal. This section will generalize Xu's results to
Hermite case. That is, we will consider the construction of the interpolation space with respect
to Hermite interpolation problem. This also proposes an approach to judge
the singularity of Hermite interpolation problem of total degree for given
$\X$ and $\bp$. 

Precisely, in this section, we consider the following interpolation problem:

Let $\X=\{X_1,X_2,\ldots,X_m\}$ be a set of pairwise distinct points in $\R^d$ and
$\bp=\{p_1,p_2,\ldots,p_m\}$ be a set of $m$ nonnegative integers. Find a subspace
$\p\subset \Pi^d$ such that for any given real numbers $c_{q,\ba},1\leq q\leq m,1\leq
|\ba|\leq p_q$ there exists a unique polynomial $f\in\p$ satisfying the interpolation
conditions
\begin{equation}
    \frac{\partial^{\alpha_1+\alpha_2+\ldots+\alpha_d}}{\partial
    x_1^{\alpha_1}\ldots \partial x_d^{\alpha_d}} f(X_q)=c_{q,\ba},\quad
    1\leq q\leq m,\quad 0\leq |\ba|\leq p_q
    \label{eq:int_condition}
\end{equation}
where $\ba=(\alpha_1,\alpha_2,\ldots,\alpha_d)\in \N_0^d$ and
$|\ba|=\alpha_1+\alpha_2+\ldots+\alpha_d$.

Following \cite{2000-Xu-p363-376}, we call such a pair $\{\X,\bp,\p\}$ correct. Clearly, such
kind of space $\p$ always exists if no any constraint is added. For the
Lagrange case, such kind of
interpolation problem was studied extensively by many authors, for example
\cite{gasca1982lagrange,2000-Gasca-p377-410,gasca2000bivariate,habib1996recursive} and the reference therein.

Before proceeding, we first present some necessary notations.
Throughout of this paper, we use the usual multi-index notation. To
order the monomials in $\Pi^d = \R[x_1,\ldots,x_d]$, we use graded
lexicographic order. Let $I$ be a polynomial ideal in $\Pi^d$. The
codimension of $I$ is denoted by $\text{codim} I$, that is,
\begin{eqnarray*}
    \text{codim} I=\dim \Pi^d/ I.
\end{eqnarray*}
If there are polynomials $f_1,f_2,\ldots,f_r$ such that every $f\in I$ can be
written as
\begin{eqnarray*}
    f=a_1f_1+a_2f_2+\ldots +a_rf_r,\quad a_j\in \Pi^d,
\end{eqnarray*}
we say that $I$ is generated by the basis $f_1,f_2,\ldots,f_r$, and
we write $I=\langle f_1,f_2,\ldots,f_r \rangle$.

For a fixed monomial order, we denote by $LT(f)$ the leading monomial term for any
polynomial $f\in \Pi^d$; that is, if $f=\sum c_{\alpha}X^{\alpha}$, then
$LT(f) =c^{\beta}X^{\beta}$, where $X^{\beta}$ is the
leading monomial among all monomials appearing in $f$.
For an ideal $I$ in $\Pi^d$ other than $\{0\}$, we denote by $LT(I)$
the leading terms of $I$, that is,
\begin{eqnarray*}
    LT (I)=\{c X^{\alpha}|\text{there exists $f\in I$ with
    $LT(f)=cX^{\alpha}$}\}.
\end{eqnarray*}
We further denote by $\langle LT(I) \rangle$ the ideal generated by
the leading terms of $LT(f)$ for all $f\in I\setminus  \{0\}$.

The following theorem is important for our purpose.
\begin{thm}[\cite{2000-Xu-p363-376}]
Fix a monomial ordering on $\Pi^d$ and let $I\in\Pi^d$ be an ideal. Then there
is an isometry between $\Pi^d/I$ and the space
\begin{eqnarray*}
    \s_I:=\span\{X^{\alpha}|X^{\alpha}\notin \langle LT(I) \rangle\}.
\end{eqnarray*}
More precisely, every $f\in\Pi^d$ is congruent modulo $I$ to a unique
polynomial $r\in \s_I$.
    \label{thm:isometry}
\end{thm}

Let $I=\langle f_1,f_2,\ldots,f_r\rangle$ and $J=\langle
g_1,g_2,\ldots,g_s\rangle$ be two polynomial ideals.
 The sum of $I$ and $J$, denoted by $I + J$, is the set
of $f+g$ where $f\in I$ and $g\in J$. The product of $I$ and $J$,
denoted by $I\cdot J$, is defined to be the ideal generated by all
polynomials $f\cdot g$ where $f\in I$ and $g\in J$. It is easy to
know that $I\cdot J=\langle f_ig_j:1\leq i\leq r,1\leq j\leq
s\rangle$. The intersection $I\cap J$ of two ideals $I$ and $J$ in
$\Pi^d$ is the set of polynomials which belong to both $I$ and $J$.
We always have $I\cdot J\subset I\cap J$. However, $IJ$ can be
strictly contained in $I\cap J$. It follows from
\cite{cox2007ideals} that if $I$ and $J$ is comaximal, then
$IJ=I\cap J$.  $I$ and $J$ is comaximal if and only if $I+J=\Pi^d$.

In
application, people usually are interested in the space with minimal degree
for fixed monomial ordering. For this purpose, consider the following polynomial ideal:
\begin{equation}
    I(\X,\bp)=\Big\{f\in \Pi^d:\frac{\partial^{\alpha_1+\alpha_2+\ldots+\alpha_d}}{\partial
    x_1^{\alpha_1}\ldots \partial x_d^{\alpha_d}} f(X_q)=0,\quad 1\leq q\leq m,\  0\leq |\ba|\leq p_q\Big\}
    \label{eq:def:ideal}
\end{equation}
If only one point $X\in\X $ then we can write it as $I(X,p)$.

\begin{thm}
    Let $I(\X,\bp)$ be the polynomial ideal defined as above. Then the
    interpolation problem satisfying Eq. \eqref{eq:int_condition} has a unique
    solution in $\s_I$.
\end{thm}
\begin{proof}
    Let $N=\sum_{i=1}^m\binom{p_i+d}{d}$. Denote the $N$ linear
    functionals in \eqref{eq:int_condition} by $F_1,F_2,\ldots,F_N$. Thus
    Eq. \eqref{eq:int_condition} can be rewritten as $F_q(f)=c_q$ for
    $1\leq q\leq N$. Suppose $n$ is an integer which is big enough such
    that there exists a polynomial $f\in \Pi_n^d$ satisfying $F_q(f)=c_q$.
    We also assume that $\s_I\subset \Pi_n^d$ although it is not necessary for
    our proof. Thus we have $\Pi_n^d=\s_I\cup (\Pi_n^d\cap I)$. Suppose
    $\varphi_1,\ldots,\varphi_t$ and $\phi_1,\ldots,\phi_s$ are the basis
    functions of $\s_I$ and $\Pi_n^d\cap I$, respectively. We further
    define a column vector
    $$\Phi=(\varphi_1,\ldots,\varphi_t,\phi_1,\ldots,\phi_s)^T:=(\Phi_1^T,\Phi_2^T)^T.$$ Since
    $n$ is big enough, we have
    $$\text{rank}(F_1(\Phi),F_2(\Phi),\ldots,F_N(\Phi))=N.$$
    Furtherly,
    $$\text{rank}(F_1(\Phi_1),F_2(\Phi_1),\ldots,F_N(\Phi_1))=N$$
    since $F_i(\Phi_2)=0$ for $1\leq i\leq N$, which leads to $t\geq N$.
    It only remains to prove $t\leq N$. If $t>N$, then
    $\mathcal{F}_i:=(F_1(\varphi_i),F_2(\varphi_i),\ldots,F_N(\varphi_i)),1\leq i\leq t$
    are linearly dependent and there exist scalars $a_1,a_2,\ldots,a_t$,
    not all zero, such that $\sum_{i=1}^ta_i\mathcal{F}_i=0$. In terms of
    the components of the vector $\mathcal{F}_i$, this shows that
    \begin{eqnarray*}
        \sum_{i=1}^ta_iF_q(\varphi_i)=0,\quad q=1,2,\ldots,N
    \end{eqnarray*}
    or,
    \begin{eqnarray*}
        F_q\Big(\sum_{i=1}^ta_i\varphi_i\Big)=0,\quad q=1,2,\ldots,N.
    \end{eqnarray*}
    The latter equations means that $\varphi=\sum_{i=1}^ta_i\varphi_i\in
    I(\X,\bp)$, which is a contradiction to $\varphi\in \s_I$ because
    every $\varphi_i\in \s_I$. Hence we have $t\leq N$ and finally
    $t=N$, which completes the proof.
\end{proof}
The theorem states that $(\X,\bp,\s_I)$ is correct. This result
was well known for Lagrange case, but less known for Hermite case.

Next, we consider the computation of $I(\X,\bp)$. If only one point is in
$\X$, the result is immediate.
\begin{lemma} \label{lem:onepoint_ideal}
    Let $X$ be a point in $\R^d$ and $p$ be a nonnegative integer, then
    \begin{equation}
        I(X,p)=\Big<l_1l_2\ldots l_{p+1}: l_i\ \text{is a linear
        polynomial vanishing at point $X$}\Big>.
        \label{eq:onepoint_ideal}
    \end{equation}
\end{lemma}
For multi-point case, we also have the similar result.
\begin{thm}\label{thm:ideal}
    \begin{eqnarray}
        \begin{split}
        I(\X,\bp)=\Big<f\in \Pi^d: & f \ \text{can be divided by the
        product of $p_i+1$} \\
        &\text{linear polynomials which pass through
        $X_i$},i=1,2,\ldots,m \Big>
    \end{split}
        \label{eq:ideal_expression}
    \end{eqnarray}
\end{thm}
\begin{proof}
    Without loss of generality, we only give a proof for $m=2$, that is,
    $\X=\{X_1,X_2\}$ and $\bp=\{p_1,p_2\}$. Let
    \begin{eqnarray*}
    &&I(X_1,p_1)=\Big\{f\in \Pi^d:\frac{\partial^{\alpha_1+\alpha_2+\ldots+\alpha_d}}{\partial
    x_1^{\alpha_1}\ldots \partial x_d^{\alpha_d}} f(X_1)=0,\quad  0\leq
    |\ba|\leq p_1\Big\},\\
    &&I(X_2,p_2)=\Big\{f\in \Pi^d:\frac{\partial^{\alpha_1+\alpha_2+\ldots+\alpha_d}}{\partial
    x_1^{\alpha_1}\ldots \partial x_d^{\alpha_d}} f(X_2)=0,\quad  0\leq
    |\ba|\leq p_2\Big\},
    \end{eqnarray*}
    Obviously $I(\X,\bp)=I(X_1,p_1)\cap I(X_2,p_2)$.
    If we denote by $J$ the right hand of Eq. \eqref{eq:ideal_expression},
    then we need to show $I(X_1,p_1)\cap I(X_2,p_2)=J$. Obviously $J\subset
    I(X_1,p_1)\cap I(X_2,p_2)$. Thus it remains to show $J\supset
    I(X_1,p_1)\cap I(X_2,p_2)$. Notice that $I(X_1,p_1) I(X_2,p_2)\subset J$
    holds. Hence the proof will be completed if we can show
    $$I(X_1,p_1)\cap
    I(X_2,p_2)=I(X_1,p_1) I(X_2,p_2).$$

    To this end, we will prove that $I(X_1,p_1)$ and $I(X_2,p_2)$ are comaximal,
    that is, $I(X_1,p_1)+I(X_2,p_2)=\Pi^d$. It is enough to show
    $1\in I(X_1,p_1)+I(X_2,p_2)$.

    Let $X_1=(x_1^{1},x_2^1,\ldots,x_d^1)$ and
    $X_2=(x_1^2,x_2^2,\ldots,x_d^2)$. Assume $x_1^1\neq x_1^2$. It is easy to check that
    $(x_1-x_1^1)^{p_1+1}\in I(X_1,p_1)$ and $(x_1-x_1^2)^{p_2+1}\in
    I(X,p_2)$. If $(x_1-x_1^1)^{p_1+1}$ and $(x_1-x_1^2)^{p_2+1}$ are seen
    as two polynomials with respect to $x_1$, that is, they are taken as
    univariate polynomials, then the greatest common divisor
    $$
    \text{gcd}((x_1-x_1^1)^{p_1+1},(x_1-x_1^2)^{p_2+1})=1,
    $$
    which means that there exist two polynomial $q_1(x_1)$ and
    $q_2(x_1)$ such that
    $$
    (x_1-x_1^1)^{p_1+1}q_1(x_1)+(x_1-x_1^2)^{p_2+1}q_2(x_1)=1.
    $$
    This complete the proof.
\end{proof}

The following two examples will be mentioned again in next section
to show the regularity of Hermite interpolation problems.

\textbf{Example 1.} Consider the case of $d=3$, $m=5$ and
$\bp=\{1,1,1,1,1\}$. Take
\begin{eqnarray*}
    X_1=(0,0,0),\ X_2=(1,0,0),\ X_3=(0,1,0),\ X_4=(0,0,1),\
    X_5=(x_0,y_0,z_0)
\end{eqnarray*}
and $\X=\{X_1,X_2,X_3,X_4,X_5\}$. With the help of Maple, the Groebner
basis of $I(\X,\bp)$ can be written as
\begin{eqnarray*}
    \{f(x,y,z,x_0,y_0,z_0),f(z,y,x,z_0,y_0,x_0),f(x,z,y,x_0,z_0,y_0),\\
    g(x,y,z,x_0,y_0,z_0),g(z,y,x,z_0,y_0,x_0),g(x,z,y,x_0,z_0,y_0),\\
    g(y,x,z,y_0,x_0,z_0),g(z,x,y,z_0,x_0,y_0),g(y,z,x,y_0,z_0,x_0),\\
    h(x,y,z,x_0,y_0,z_0),h(y,x,z,y_0,x_0,z_0),h(y,z,x,y_0,z_0,x_0),\\
    w(x,y,z,x_0,y_0,z_0),w(y,z,x,y_0,z_0,x_0),w(x,z,y,x_0,z_0,y_0)\},
\end{eqnarray*}
where
\begin{eqnarray*}
&&f(x,y,z,x_0,y_0,z_0)=-x_0z_0(z_0-1)(x_0z_0+l(X_5)-y_0)(zy^2+z^2y-zy)\\
&&\qquad\quad+2(z_0-1)z_0(y_0z_0^2+x_0z_0^2-y_0z_0+x_0y_0z_0-x_0z_0+x_0y_0)xyz\\
&&\qquad\quad-y_0(z_0-1)z_0(y_0z_0+l(X_5)-x_0)(xz^2+x^2z-xz)\\
&&\qquad\quad+x_0y_0l(X_5)(z^4-2z^3+z^2)-z_0^2(z_0-1)^3(x^2y+xy^2-xy),\\
&&g(x,y,z,x_0,y_0,z_0)=x_0z_0(x_0z_0+z_0-1)(zy-zy^2)+z_0^2(z_0-1)^2(xy-xy^2-x^2y)\\
&&\qquad\quad+z_0(-z_0^2+2y_0z_0^2+2x_0z_0^2+2x_0y_0z_0+2z_0-3y_0z_0-4x_0z_0-1+y_0+2x_0)xyz\\
&&\qquad\quad+x_0l(X_5)z^3y-x_0(x_0z_0^2+x_0+y_0+z_0^2-1)z^2y
+(xz-x^2z-xz^2)y_0^2z_0^2\\
&&h(x,y,z,x_0,y_0,z_0)=2y_0z_0(y_0z_0-y_0+x_0y_0-z_0+x_0z_0+1-2x_0)xyz\\
&&\qquad\quad-y_0z_0^2(z_0-1)(xy^2+x^2y-xy)
-y_0^2z_0(y_0-1)(xz^2+x^2z-xz)\\
&&\qquad\quad-x_0y_0z_0(1+x_0)(zy^2+z^2y-zy)
+x_0l(X_5)z^2y^2\\
&&w(x,y,z,x_0,y_0,z_0)=l(X_5)xy^2z+(y_0^2-y_0^3)(xz^2+x^2z-xz)-x_0^2y_0(zy^2+z^2y-yz)\\
&&\qquad\quad-y_0z_0^2(xy^2+x^2y-xy)+y_0(2x_0y_0+2y_0z_0+2x_0z_0-3x_0-2y_0-3z_0+2)xyz
\end{eqnarray*}
and $l(X_5)=x_0+y_0+z_0-1$. It is easy to see that if any four nodes do not lie on
hyperplane then $S_I=\Pi_3^3$. Hermite interpolation of type total degree is affinely
invariant in
    the sense that if the interpolation is singular or regular. Hence if
    the given five nodes are in general position, that is, no four nodes
    lie on a hyperplane, then any four of them can be transformed into
    $X_1,X_2,X_3,X_4$. This example implies that uniform Hermite interpolation of
    type total degree on 5 nodes up to order 1 in $\R^3$ is almost
    regular.

    \textbf{Example 2.} Consider the case of $d=3,m=6$ and
    $\bp=\{3,3,3,3,3,3\}$. Take
    \begin{eqnarray*}
        X_1=(0,0,0),\ X_2=(1,0,0),\ X_3=(0,1,0),\ X_4=(0,0,1),\
        X_5=(1,1,1),\ X_6=(2,1,1)
    \end{eqnarray*}
    and $\X=\{X_1,X_2,X_3,X_4,X_5,X_6\}$. With the help of Maple, we have
    $S_I=\Pi_7^3$. Thus uniform Hermite interpolation of type total degree on 6 points
    up to order 3 in $\R^3$ is almost regular.

\section{Singular Interpolation Schemes}
In this section, we will investigate Hermite interpolation of type total degree which is
singular. Our results covers those appeared in \cite{1992-Lorentz-p-}, but the method employed here is
basically different. In \cite{1992-Lorentz-p-}, most of the results are proved logically
and many complicated inequalities are employed. For comparison, our method
basically depends on an algorithm to be developed (see Theorem \ref{thm:main}) which implies
that it is a construction method. One will find the polynomial
satisfying the homogenous condition by Theorem \ref{thm:main}.
%

In this section, Eq. \eqref{number_condition} is always assumed to
hold. In this case, the interpolation space and the set of
functionals to be interpolated are affinely invariant.  Furthermore, throughout
this paper, we always assume that $d\geq 2$ and $m\geq 2$. That is, we
deal with multi-point Hermite interpolation problem in several variables.

The following theorem and corollary will give an evaluation of $n$ in Eq.
\eqref{number_condition}.
\begin{thm}
    Given $\X=\{X_1,X_2,\ldots,X_m\}$ and $\bp=\{p_1,p_2,\ldots,p_m\}$, if
    there exists an $n$ such that $(\X,\bp,\Pi_n^d)$ correct, then the
    following inequality holds:
    \begin{eqnarray}\label{eq:num_ineq}
        \binom{n+\tilde{d}}{\tilde{d}}\geq
        \sum_{i=1}^{\tilde{d}+1}\binom{p_{q_i}+\tilde{d}}{\tilde{d}}
    \end{eqnarray}
    where $1\leq \tilde{d}\leq d$ and the right side denotes
    the sum of any $\tilde{d}+1$ terms. If $m<\tilde{d}+1$, we assume
    \begin{eqnarray*}
        \sum_{i=1}^{\tilde{d}+1}\binom{p_{q_i}+\tilde{d}}{\tilde{d}}=\sum_{i=1}^{m}\binom{p_i+\tilde{d}}{\tilde{d}}
    \end{eqnarray*}
    \label{thm:ineq}
\end{thm}
\begin{proof}
    Assume $m\geq\tilde{d}+1$. Note that Eq.
    \eqref{number_condition} holds since $(\X,\bp,\Pi_n^d)$ is correct.
    Thus inequality \eqref{eq:num_ineq} trivially holds for $\tilde{d}=d$.

    Consider the case of $\tilde{d}<d$. Suppose
    $X_{q_1},X_{q_2},\ldots,X_{q_{\tilde{d}+1}}$ are $\tilde{d}+1$
    arbitrary nodes. Then there exist $d-\tilde{d}$ linearly independent
    linear polynomial $l_i(X),i=\tilde{d}+1,2,\ldots,d$ vanishing on these
    $\tilde{d}+1$ nodes. Assume $l_i(X),i=1,\ldots,\tilde{d}$ are
    $\tilde{d}$ linear polynomial such that all $l_i(X),i=1,2,\ldots,d$
    are linearly independent. Take the following affine transformation
    \begin{eqnarray}
        T:\quad y_i=l_i(X),i=1,2,\ldots,d
        \label{eq:transform}
    \end{eqnarray}
    Let $Y_i=T(X_i),i=1,2,\ldots,m$ and $\Y=T(\X)$. Thus under the new coordinate system,
    the last $d-\tilde{d}$ coordinates of
    $Y_{q_i},i=1,2,\ldots,\tilde{d}+1$ are zero.

    Hermite interpolation of type total degree is affinely invariant in
    the sense that if the interpolation is singular or regular. Hence
    $(\Y,\bp,\Pi_n^d)$ is also correct. Thus for any given $\{c_{q,\ba},0\leq|\ba|\leq p_q,1\leq q\leq m\}$
    there is a unique $f\in \Pi_n^d$ satisfying
    \begin{equation}
    \frac{\partial^{\alpha_1+\alpha_2+\ldots+\alpha_d}}{\partial
    y_1^{\alpha_1}\ldots \partial y_d^{\alpha_d}} f(Y_q)=c_{q,\ba},\quad
    1\leq q\leq m,\quad 0\leq |\ba|\leq p_q
    \end{equation}
    Specially, we have
    \begin{equation*}
        \frac{\partial^{\alpha_1+\alpha_2+\ldots+\alpha_{\tilde{d}}}}{\partial
        y_1^{\alpha_1}\ldots \partial y_{\tilde{d}}^{\alpha_{\tilde{d}}}} f(Y_q)=c_{q,\ba},\quad
        1\leq q\leq m,\quad 0\leq
        |\ba|=\alpha_1+\alpha_2+\ldots+\alpha_{\tilde{d}}\leq p_q
    \end{equation*}
    which means
    \begin{eqnarray*}
        \binom{n+\tilde{d}}{\tilde{d}}\geq
        \sum_{i=1}^{\tilde{d}+1}\binom{p_{q_i}+\tilde{d}}{\tilde{d}}
    \end{eqnarray*}
    If $m<\tilde{d}+1$, then we select all $m$ nodes instead of
    $\tilde{d}+1$ arbitrary nodes. In this case, the proof is straightward. 
    The proof is completed.
\end{proof}

For convenience, we always order $0\leq p_1\leq p_2\leq \ldots\leq
p_m$ in what follows.
\begin{corollary}\label{cor:ineq2}
Assume $m\geq 2$. If set $\tilde{d}=1$, then
\begin{eqnarray}
    n+1\geq p_m+p_{m-1}+2, \quad \text{or}\quad n\geq p_m+p_{m-1}+1.
    \label{eq:ineq2}
\end{eqnarray}
\end{corollary}

\begin{rem}\label{rem}
    Obviously, if \eqref{eq:ineq2} does not hold, then the interpolation
    scheme must be singular. In what follows, we can prove the singularity
    of one interpolation scheme by assuming \eqref{eq:ineq2}.
\end{rem}

The following theorem can be use to judge whether the interpolation scheme is singular
for small $m$.
\begin{thm}\label{thm:main}
    Assume $m\geq 2$. Given $\X=\{X_1,X_2,\ldots,X_m\}$ and $\bp=\{p_1,p_2,\ldots,p_m\}$, if
    \begin{eqnarray}
        p_1+p_2+\ldots+p_m+m\leq nd,
        \label{eq:ineq3}
    \end{eqnarray}
    then Hermite interpolation of type
    total degree is singular. Here the numbers $p_q$ and $n$ are assumed
    to satisfy Eq. \eqref{number_condition}.
\end{thm}
\begin{proof}
    We only need to find a polynomial satisfying the homogenous
    interpolation condition, which can be done by giving an algorithm for
    its construction.

    \begin{quote}
    {\bf Step 1.} Set $f(x_1,x_2,\ldots,x_d)=1$ and $\tilde{\bp}=\{\tilde{p}_1,\tilde{p}_2,\ldots,\tilde{p}_m\}:=\{p_1+1,p_2+1,\ldots,p_m+1\}$.

    {\bf Step 2.} If the number of the nonzero in $\tilde{\bp}$ is no more
    than $d$, let $l(x_1,x_2,\ldots,x_d)$ be a linear
    polynomial vanishing on $X_1,X_2,\ldots,X_m$. Take
    $\tilde{p}_{\max}=\max\{\tilde{p}_1,\tilde{p}_2,\ldots,\tilde{p}_m\}$ and set $f=f\cdot
    l^{\tilde{p}_{\max}}$ and stop. Otherwise,
    go to step 3.

    {\bf Step 3.} Suppose
    $\tilde{p}_{i_1},\tilde{p}_{i_2},\ldots,\tilde{p}_{i_d}$ are $d$ largest numbers in
    $\tilde{\bp}$. Clearly, there must exist at least one linear polynomial
    vanishing at any $d$ points. Denote by $l_{i_1i_2\ldots i_d}$ the
    linear polynomial vanishing on $X_{i_1},X_{i_2},\ldots,X_{i_d}$.
    Set $f=f\cdot l_{i_1i_2\ldots i_d}$.
     Let $\tilde{p}_{i_j}=\tilde{p}_{i_j}-1, j=1,2,\ldots,d$ and
    $\tilde{p}_{i}=\tilde{p}_i$ if $i\in \{1,2,\ldots,m\}/\
    \{i_1,i_2,\ldots,i_d\}$. Go to Step 2.
\end{quote}

    We want to show that the polynomial $f$ constructed by this algorithm
    satisfied our requirement. Firstly, to this end we need to show that the algorithm does eventually terminate.
    Denote $|\tilde{\bp}|=\sum_{i=1}^m\tilde{p}_i$. The key observation is
    that $|\tilde{\bp}|$ will be dropped by $d$ after step 3. Hence the
    algorithm will terminate since at the beginning $|\tilde{\bp}|\leq nd$.

    Next, it is easy to know that this kind of polynomial $f$ constructed
    by the algorithm satisfies the homogenous interpolation conditions by
    Theorem \ref{thm:ideal}.

    Finally, we need to show that $\deg(f)\leq n$.
    Assume that Step 3
    has been run $t$ times totally. Clearly, $t$ is no more than $n$ according to
    the assumption of the theorem. To complete the proof, now we estimate
    the degree of the polynomial $f$.  Suppose we are in a situation to run the final step. That is, there
    are only
    no more than $d$ nonzero numbers in $\tilde{\bp}$. We consider the following three
    cases.
    \begin{itemize}
    \item[1)] If all the numbers in $\tilde{\bp}$ are zeros, then the degree
        of $f$ is $t$ which is not larger than $n$.
    \item[2)] The largest number in $\tilde{\bp}$ equals to 1, that is,
        $\tilde{p}_{\max}=1$. In this case, we have $t\leq n-1$, which will
        leads to $\deg(f)\leq n$.
    \item[3)] If $\tilde{p}_{\max}>1$, we will show that the degree of $f$ is no
        more than $k:=\max\{p_1+1,p_2+1,\ldots,p_m+1\}$. Clearly, it is enough to show that $\tilde{p}_{max}+t=k$.
        To this purpose, denote by $\tilde{\bp}^{(j)}$ the set
        $\tilde{\bp}$ after running the third Step $j$ times. For
        convenience, we also write $\{p_1+1,p_2+1,\ldots,p_m+1\}$ as
        $\tilde{\bp}^{(0)}$. Based on these notation we have that
        $\tilde{p}_{\max}$ is the maximum number in
        $\tilde{\bp}^{(t)}$. Furthermore, there are at most $d$
        numbers in $\tilde{\bp}^{(t)}$ different from zero. Thus we
        can deduce that $\tilde{p}_{\max}+1$ must be the maximum
        number in $\tilde{\bp}^{(t-1)}$. If it is not the case, then
        there must exist $d$ numbers in $\tilde{\bp}^{(t-1)}$ are
        larger than or equals to $\tilde{p}_{\max}$, which will leads
        to more than $d$ numbers in $\tilde{\bp}^{(t)}$ different from
        zero and furtherly contradict the conclusion above. By a
        similar discussion, we finally derive that
        $\tilde{p}_{\max}+t$ is the maximum number in
        $\tilde{\bp}^{(0)}$ which implies that $\tilde{p}_{\max}+t=k$.
    \end{itemize}
    By collecting above discussion, we get that $f$ satisfies the
    interpolation condition and its degree is no more than $n$. Thus the
    interpolation scheme is singular, which completes the proof.
\end{proof}
\begin{corollary}\label{cor:hompoly}
    Given $\X=\{X_1,X_2,\ldots,X_m\}$ and $\bp=\{p_1,p_2,\ldots,p_m\}$, if
    Eq. \eqref{eq:ineq3} holds
    then there exists a polynomial of degree $\leq n$, together with all of
    its partial derivatives of order up to
    $p_i$, vanishing at $X_i$ for all $i$.
\end{corollary}

Theorem \ref{thm:main} is sharp for small $m$. By corollary \ref{cor:ineq2}
and theorem \ref{thm:main}, we have the following result which appeared in
\cite{1992-Lorentz-p-}. As comparison, several different cases were considered
and a lot of complicated inequalities were employed there. But here, all the
different cases are dealt with uniformly and the proof is very short.

\begin{thm}\label{thm:k=1}
    All Hermite interpolation of type total degree are singular in $\R^d$ with
    $d\geq 2$ if the number $m$ of nodes satisfies $2\leq m\leq d+1$,
    except for Lagrange case.
\end{thm}
\begin{proof}
    Acorrding to Remark \ref{rem}, we assume $n\geq p_m+p_{m-1}+1$. We only
    need to show that Eq. \eqref{eq:ineq3} holds in this case. It is easy to
    derive that
    \begin{eqnarray*}
        p_1+p_2+\ldots+p_m+m&\leq& p_1+p_2+\ldots+p_m+d+1\\
        &\leq &d p_{m-1}+p_m+d+1\\
        &\leq &d(p_m+p_{m-1}+1)\\
        &\leq &nd.
    \end{eqnarray*}
    This completes the proof by Theorem \ref{thm:main}.
\end{proof}
For true Hermite interpolation, $p_m\geq 1$. Thus we have the following
general result.
\begin{corollary}\label{cor:d+k}
   Assume $m\geq 2$. All Hermite interpolation of type total degree are singular in $\R^d$ with
    $d\geq k(1+\frac{1}{p_m})$ if $m\leq d+k$, except for Lagrange case.
\end{corollary}
\begin{proof}
    Acorrding to Remark \ref{rem}, we assume $n\geq p_m+p_{m-1}+1$. We only need to show that Eq. \eqref{eq:ineq3} holds in this case.
    Thus,
    \begin{eqnarray*}
        p_1+p_2+\ldots+p_m+m&\leq& p_1+p_2+\ldots+p_m+d+k\\
        &\leq &(d+k-1) p_{m-1}+p_m+d+k\\
        &\leq &dp_{m-1}+kp_m+d+k\\
        &= & dp_{m-1}+k(p_m+1)+d\\
        &=& dp_{m-1}+kp_m(1+\frac{1}{p_m})+d \\
        &\leq & dp_{m-1}+dp_m+d\\
        &\leq & nd.
    \end{eqnarray*}
This completes the proof.
\end{proof}
 $p_m\geq 1$ which means
$1+\frac{1}{p_m}\leq 2$. Thus we have
\begin{corollary}\label{cor:d+k:2}
    Assume $m\geq 2$. All Hermite interpolation of total degree are singular in $\R^d$ with
    $d\geq 2k$ if the number $m$ of nodes satisfies $m\leq d+k$, except for
    Lagrange case.
\end{corollary}
From corollary \ref{cor:d+k:2}, we know that all Hermite
interpolation of total degree are singular with $d\geq 4$ if the
number of nodes satisfies $m\leq d+2$. And also from corollary
\ref{cor:d+k}, all Hermite interpolation of total degree are
singular for $d=3$ and $m=5$ if $p_m\geq 2$. We claim that this
result is very sharp because the corresponding Hermite interpolation
with $d=3,m=5$ and $p_1=p_2=p_3=p_4=p_5=1$ is almost regular. The
regularity was proved by the method of determinant in
\cite{1992-Lorentz-p-} and also be shown by example 1 in section 2.
Besides, if $p_m\leq 1$ and $p_1=0$, Eq. \eqref{number_condition}
never holds. Therefore, for $d=3$, only one case can produce regular
interpolation scheme.

Now let us consider the case of $d=2$ and
$m=d+2$. It is well known, interpolating the value of a function and all of its
    partial derivatives of order up to $p$ at each of the three vertices
    of a triangle as well as the value of the function and all of its
    derivatives of order up to $p+1/p-1$ at a fourth point lying anywhere
    in the interior of the triangle by polynomials from $\Pi_{2p+2}^2/\Pi_{2p+1}^2$ is
    regular. We will prove that
    in all other cases, the corresponding Hermite interpolation scheme is
    singular. This can be done by Lemmas \ref{thm:m=4} and
    \ref{thm:m=42}.

    \begin{lemma}\label{thm:m=4}
    All Hermite interpolations of type total degree on $4$ points in
    $\R^2$ are singular if $0\leq p_1\leq p_2\leq p_3\leq p_4$, $p_4>p_2$ and
    $p_3>p_1$.
\end{lemma}
\begin{proof}
    Suppose the interpolation problem is regular, that is, there is a
    unique $f(x_1,x_2)\in \Pi_n^2$ satisfying
    \begin{equation}\label{int_4node}
    \frac{\partial^{\alpha_1+\alpha_2}}{\partial
    x_1^{\alpha_1}\partial x_2^{\alpha_2}} f(X_q)=c_{q,\ba},\quad
    1\leq q\leq 4,\quad 0\leq |\ba|\leq p_q.
    \end{equation}
    It follows from corollary \ref{cor:ineq2} that $n\geq p_3+p_4+1$.

    Denote by $l_{ij}=0$ the line passing through $X_i$ and
    $X_j$.  Consider the following polynomial
    \begin{eqnarray*}
        f=l_{12}^{p_1+1}l_{34}^{p_3+1}l_{24}^{\max\{p_4-p_3,p_2-p_1\}}.
    \end{eqnarray*}
    Clearly, $f$ satisfies the homogenous interpolation conditions in
    Eq. \eqref{int_4node} and has the degree of
    \begin{eqnarray*}
        (p_1+1)+(p_3+1)+\max\{p_4-p_3,p_2-p_1\}=\max\{p_1+p_4,p_2+p_3\}+2.
    \end{eqnarray*}
    To complete the proof, it remains to prove $n$ is not less than the
    degree of $f$, that is,
    \begin{eqnarray*}
        n\geq \max\{p_1+p_4,p_2+p_3\}+2.
    \end{eqnarray*}
    According to the discussion above, it is enough to show
    \begin{eqnarray*}
        p_4+p_3+1\geq \max\{p_1+p_4,p_2+p_3\}+2
    \end{eqnarray*}
    which is equivalent to
    \begin{eqnarray}\label{eq:lem:inq}
        (p_4-p_2)+(p_3-p_1)\geq |(p_4-p_2)-(p_3-p_1)|+2.
    \end{eqnarray}
    Inequality \eqref{eq:lem:inq} will be satisfied if $p_4>p_2$ and
    $p_3>p_1$. This completes the proof.
\end{proof}
\begin{lemma}\label{thm:m=42}
    All Hermite interpolation of total degree with $d=2$ and $m=4$ are
    singular if the orders satisfy one of the two conditions
    \begin{eqnarray*}
        i).\quad &&p_1=p_2=p_3\ \text{and} \ p_4\geq p_1+2,\\
        ii).\quad &&p_2=p_3=p_4 \ \text{and}\ p_4\geq p_1+2.
    \end{eqnarray*}
\end{lemma}
\begin{proof}
    We only give a proof for case i). In this case we have
    \begin{eqnarray*}
        p_1+p_2+p_3+p_4+4&=&3p_3+p_4+4\\
        &\leq & 2p_3+2p_4+2\\
        &=&2(p_3+p_4+1)\leq nd
    \end{eqnarray*}
    which completes the proof by theorem \ref{thm:main}.
\end{proof}

It is well known that uniform Hermite interpolation of type total degree never
happen because Eq. \eqref{number_condition} in this case does not hold, see
\cite{1992-Lorentz-p-}. Thus for $m=d+2$, we have
\begin{thm}\label{thm:k=2}
  Consider the problem of Hermite interpolation of type total degree
  on $m=d+2$ nodes in $\R^d$. Then
  \begin{itemize}
    \item For $d=2$, if $p_1=p_2=p_3, p_4=p_3+1$ or $p_1=p_2-1,
    p_2=p_3=p_4$, it is almost regular.
    \item For $d=3$, if $p_1=p_2=p_3=p_4=p_5=1$ and $n=3$, it is
    almost regular.
    \item Otherwise, it is singular.
  \end{itemize}
\end{thm}

Let us consider the case of $m=d+3$. From corollary \ref{cor:d+k},
all Hermite interpolation of total degree are singular in one of
the following cases: i) $d\geq 6,m=d+3$ and $p_m\geq 1$; ii) $d\geq
5,m=d+3$ and $p_m\geq 2$; iii) $d\geq 4,m=d+3$ and $p_m\geq 3$.

For the case of $d=5,m=8$ and $p_m=1$, it is easy to check Eq.
\eqref{number_condition} never holds.

Let us turn to the case of $d=4,m=7$ and $p_m\leq 2$. Eq.
\eqref{number_condition} holds only if i) $p_6=p_7=2,p_i=0,1\leq
i\leq 5$ and $n=3$; ii) $p_6=p_7=1,p_i=0,1\leq i\leq 5$ and $n=2$.
For these two interpolation schemes, we need the following result
from \cite{sauer1995multivariate}:
\begin{thm}\label{xu}
Multivariate Hermite interpolation of type total degree (4) in $R^d$
with at most $d+1$ nodes having $p_q\geq 1$ is regular a.e. if and
only if
$$
p_q+p_r< n
$$
for $1\leq q,r\leq m,q\neq r$.
\end{thm}

Obviously the interpolation problems considered above are singular by the
theorem above.

Consider the case of $d=3$ and $m=6$. According to Corollary \ref{cor:ineq2},
$n\geq p_6+p_5+1$. Thus if $p_6-2\geq p_5=p_4=p_3=p_2=p_1=p$, then
\begin{eqnarray*}
	nd\geq 3(p_6+p_5+1)&=& p_6+2p_6+3p+3\\
	&\geq & p_6+5p+7>p_1+p_2+\ldots+p_6+6
\end{eqnarray*}
which means that the Hermite interpolation problem is singular by  Theorem
\ref{thm:main}.
If $p_6-1=p_5=p_4=p_3=p_2=p_1=p$, then $n\geq 2p+2$
and
\begin{eqnarray*}
    \sum_{i=1}^6\binom{p_i+3}{3}-\binom{n+3}{3}&\leq&
    5\binom{p+3}{3}+\binom{p+4}{3}-\binom{2p+2+3}{3}\\
    &=& -\frac{1}{6}(2p+3)(p+2)(p+1)< 0
\end{eqnarray*}
which implies that Eq. \eqref{number_condition} never holds.
If $p_6>p_5>p_1$, then
\begin{eqnarray*}
    p_1+p_2+\ldots+p_6+6&\leq& 5p_5+p_6+5\\
    &\leq & 3p_5+3p_6+3\\
    &\leq & 3(p_5+p_6+1)\leq 3n.
\end{eqnarray*}
Thus it follows from Theorem \ref{thm:main} that the Hermite
interpolation problem is also singular in this case.

Otherwise, $p_5=p_6$. In this case
\begin{eqnarray*}
    \sum_{i=1}^6\binom{p_i+3}{3}-\binom{n+3}{3}&\leq&
    6\binom{p_6+3}{3}-\binom{2p_6+4}{3}\\
    &=& -\frac{1}{3}(p_6-3)(p_6+2)(p_6+1).
\end{eqnarray*}
Thus Eq. \eqref{number_condition} does not hold for $p_5=p_6>3$.
Moreover, by a careful check and computation,  Eq.
\eqref{number_condition} also does not hold for $p_5=p_6<3$. As a
result, Eq. \eqref{number_condition} only hold for $p_i=3$ and
$n=7$. By Example 2 in section 2, the uniform Hermite interpolation
problem is almost regular.

Finally, we consider the case of $d=2$ and $m=5$. In this case, we claim that
$n<2p_5+2$ if the Hermite interpolation problem is regular. In fact, for any 5 points in the plane, there must exist a
non-trival quadratic $Q(x,y)$ vanishing at these points. Let
$P(x,y)=[Q(x,y)]^{p_5+1}$. Then $P$, together with all of its partial
derivatives of order up to $p_5$, vanish at these points.

\begin{lemma}\label{lem:singula5}
    Given $\X=\{X_1,X_2,\ldots,X_5\}\subset \R^2$ and
    $\bp=\{p_1,p_2,\ldots,p_5\}$ with $p_1\leq p_2\leq p_3\leq p_4\leq p_5$,
    if Eq. \eqref{number_condition} holds and
    \begin{eqnarray}\label{eq:5points}
    p_2+p_3+2\leq p_4+p_5
    \end{eqnarray}
    then Hermite interpolation of type
    total degree is singular.
\end{lemma}
\begin{proof}
    As discussed above, there exists a non-trival quadratic $Q(x,y)$
    vanishing at these five points. Since
    \begin{eqnarray*}
        &&(p_2-p_1-1)+(p_3-p_1-1)+(p_4-p_1-1)+(p_5-p_1-1)+4\\
        &=&p_2+p_3+p_4+p_5-4p_1\leq 2(p_4+p_5-2p_1-1),
    \end{eqnarray*}
    due to Corollary \ref{cor:hompoly} there exists a polynomial
    $f(x,y)$ of degree $\leq p_4+p_5-2p_1-1$, together with all of its
    partial derivatives of order up to
    $p_i-p_1-1$(if negative, no interpolation happens), vanishing at $X_i$ for
    all $2\leq i\leq 5$. Let $P(x,y)=[Q(x,y)]^{p_1+1}\cdot f(x,y)$. Thus  $P$, together with all of its partial
derivatives of order up to $p_i$, vanish at $X_i$ for all $i$. It is easy to
get that the degree of $P$ is no more than $p_4+p_5+1$. This completes the proof.
\end{proof}

\begin{lemma}[\cite{1992-Lorentz-p-}]
    The uniform Hermite interpolation of type total degree on $5$ nodes in
    $\R^2$ is singular.
\end{lemma}
\begin{lemma}
    Assume $d\geq 2$. If $p_1\leq p_2+1=p_3=p_4=p_5$ and Eq. \eqref{number_condition} holds,
    then the Hermite interpolation of type total degree is singular.
\end{lemma}
\begin{proof}
    Suppose the interpolation problem is regular. Then, $p_4+p_5+1\leq n
    <2p_5+2$, that is, $n=2p_5+1$. However,
    \begin{eqnarray*}
        \sum_{i=1}^5\binom{p_i+2}{2}>\sum_{i=2}^5\binom{p_i+2}{2}=\binom{n+2}{2},
    \end{eqnarray*}
    which contradicts Eq. \eqref{number_condition}.
\end{proof}
\begin{lemma}
    Assume $d\geq 2$. If $p_1\leq p_2=p_3=p_4=p_5-1$ and Eq. \eqref{number_condition} holds,
    then the Hermite interpolation of type total degree is singular.
\end{lemma}
\begin{proof}
    Suppose the interpolation problem is regular. Then we have
    $$p_4+p_5+1\leq
    n<2p_5+2
    $$
    and
    \begin{eqnarray*}
        \sum_{i=1}^5\binom{p_i+2}{2}>\sum_{i=2}^5\binom{p_i+2}{2}=\binom{2p_4+2+2}{2}.
    \end{eqnarray*}
    Thus, $n=2p_5+1$. Let  $Q(x,y)$ be the quadratic polynomial vanishing
    at these 5 points. Again by the way of Lemma
    \ref{lem:singula5}, we can get a polynomial of degree no more than
    $n-(2p_1+2)$, together with all of its partial derivatives of order up
    to $p_i-p_1-1$, vanishing at $X_i$ for $2\leq i\leq 5$. Thus
    $[Q(x,y)]^{p_1+1}\cdot f$ satisfies the homogenous interpolation
    condition. It is easy to check that
    \begin{eqnarray*}
        2(p_1+1)+n-(2p_1+2)=n
    \end{eqnarray*}
    which completes the proof.
\end{proof}

By collecting the discussion above, we have
\begin{thm}\label{thm:k=3}
    Assume $d\geq 2$. All Hermite interpolation of total degree on $m=d+3$ points are
    singular in $\R^d$  except for the case of $d=3,n=7$ and $p_i=3$ for
    $i=1,2,\ldots,6.$
\end{thm}

Combing Corollary \ref{cor:d+k:2} and Theorems
\ref{thm:k=1},\ref{thm:k=2},\ref{thm:k=3}, we have proved
\begin{thm}
  Assume $d\geq 2$. Consider the problem of Hermite interpolation of type total degree on $m=d+k$
  nodes in $\R^d$. Then
  \begin{enumerate}
    \item For $k\leq 1$, it is singular.
    \item For $k=2$, if $d=2$, and $p_1=p_2=p_3, p_4=p_3+1$ or $p_1=p_2-1,
    p_2=p_3=p_4$, it is almost regular; else if $d=3$, and $p_i=1,i=1,2,3,4,5,n=3$, it is almost regular; otherwise it is singular.
    \item For $k=3$, if $d=3$ and $p_i=3,(i=1,2,\ldots,6), n=7$, it
    is almost regular; otherwise it is singular.
    \item If $d\geq 2k$, it is singular.
  \end{enumerate}
\end{thm}

\section{Conclusion}
In this paper, we consider the singular problem of multivariate
Hermite interpolation of total degree. We make a detailed
investigation for Hermite interpolation problem of type total degree
on $m=d+k$ nodes in $\R^d$. Our results imply that the interpolation
problem in $\R^d$ is singular on $m=d+k$ nodes for $d\geq 2k$. For
$k\leq 3$, it is shown that only very few cases can produce regular
interpolation. The method developed in this paper can deal with the
case of small $k$ compared to $d$. For bigger $k$, Theorem
\ref{thm:main} is not very sharp.

\bibliographystyle{plain}

\end{document}